\documentclass[11pt]{amsart}
\usepackage{amsmath, amssymb, amsfonts,enumerate}
\usepackage[all]{xy}
\usepackage{amscd}
\usepackage{comment}
\usepackage{mathtools}
\usepackage{tikz} 
\usepackage{tikz-cd} 
\usepackage{hyperref}

 \newtheorem{theorem}{Theorem}[section]
\newtheorem{lemma}[theorem]{Lemma}
\newtheorem{corollary}[theorem]{Corollary}
\newtheorem{proposition}[theorem]{Proposition}
 
\newtheorem*{claim*}{Claim} 
 
 \theoremstyle{definition}
 \newtheorem{definition}[theorem]{Definition}
 \newtheorem{remark}[theorem]{Remark}

 \newtheorem{example}[theorem]{Example}

\newtheorem*{step*}{Step*}  

\numberwithin{equation}{section}
\newcommand {\N}{\mathbb{N}} 
\newcommand {\Z}{\mathbb{Z}} 
\newcommand {\C}{\mathbb{C}} 

\newcommand{\HH}{\mathcal{H}}

\DeclareMathOperator{\Ker}{Ker}
\DeclareMathOperator{\End}{End}

\begin{document}
\title[Shadowing property and endomorphisms of group shifts]{Shadowing for families of endomorphisms of generalized group shifts}  
\author[X.K.Phung]{Xuan Kien Phung}
\address{} 
\email{phungxuankien1@gmail.com}
\begin{abstract}
Let $G$ be a countable monoid and let $A$ be an Artinian group (resp. an Artinian module). 
Let $\Sigma \subset A^G$ be a closed subshift which is also a  subgroup (resp. a submodule) of $A^G$. 
Suppose that $\Gamma$ is a finitely generated monoid consisting of pairwise commuting cellular automata $\Sigma \to \Sigma$
that are also homomorphisms of groups (resp. homomorphisms of modules) with monoid binary operation given by composition 
of maps.  
We show that the   valuation action of $\Gamma$ on $\Sigma$ satisfies a natural intrinsic shadowing property. 
Generalizations are also established for families of endomorphisms of admissible group subshifts.   
\end{abstract}
\date{\today}
\maketitle
 
\setcounter{tocdepth}{1}
 \section{Introduction}  
The notion of pseudo-orbit arises in dynamical systems with noise and dates back at least to 
Birkhoff \cite{birkhoff-pseudo-orbit}. Given 
$\delta >0$, a $\delta$-pseudo-orbit of a dynamical system can be thought
of as an orbit computed by computers with an inherent error uniformly bounded by $\delta$ at every step of calculation. 
The remarkable shadowing property (cf.~Anosov \cite{anosov-geodeisc-riemann}, Bowen \cite{bowen-shadow-equilibrium}), also known as the pseudo-orbit tracing property, says that every pseudo-orbit can be approximated by an exact orbit. 
The fundamental shadowing lemma due to Bowen \cite{bowen-shadow-equilibrium}   states 
that a diffeomorphism of a compact manifold has the shadowing property on its hyperbolic set. 
\par 
Important consequences, with deep connections to symbolic dynamics, of group actions on compact metric spaces that satisfy the shadowing property have been discovered notably by Walters  \cite{walters-shadow}, KurKa \cite{kurka-97}, Blanchard and Maass \cite{blanchard-maass-97}, and more recently by Chung and Lee \cite{chung-shadow}, and  Meyerovitch  \cite{meyerovitch-pseudo-orbit}.
Let $G$ be a finitely generated group and let $A$ be a nonempty 
finite set.  Chung and Lee \cite{chung-shadow} show that a closed subshift $\Sigma$ of $A^G$
is of finite type if and only if the shift action of $G$ on  $\Sigma$ has the shadowing property, which extends results of Walters \cite{walters-shadow} and Oprocha \cite{oprocha-08}. 
They also prove that every
expansive continuous action of a countable group with the shadowing property on a compact metric space satisfies the topological stability. Moreover, 
Meyerovitch \cite{meyerovitch-pseudo-orbit} shows that  
an expansive continuous action 
of a countable amenable group  on a compact metric space with positive entropy 
must admit off-diagonal asymptotic pairs if the action has the shadowing property. 
The recent work of 
Barbieri and Ramos-Garcia and Li  \cite{li-2020-markov} also  investigates 
various weak shadowing properties and their consequences for the well-studied class of algebraic actions. 
\par 
Therefore, it is interesting to establish classes of group actions that satisfy the shadowing property. Nontrivial classes are obtained by Chung and Lee \cite{chung-shadow} for  equicontinuous actions (e.g. distal actions) of an infinite, finitely generated group on the Cantor space and by   
Meyerovitch \cite{meyerovitch-pseudo-orbit} for  
expansive principal algebraic actions of a countable group on a compact metrizable abelian group. Osipov and Tikhomirov in  \cite{tikhomirov-14}, and 
Pilyugin and Tikhomirov in  \cite{tikhomirov-03} also obtain some sufficient conditions for group actions on metric spaces to have the shadowing property. 
\par 
Nevertheless, it seems difficult in general to construct concrete and nontrival examples in the above mentioned classes of actions. Moreover, continuous group actions must act by homeomorphisms 
while the notion of shadowing property extends naturally to many interesting non-homeomorphic maps.  
Our main goal in this paper is to show that  
a natural intrinsic shadowing property (Definition~\ref{d:shadow-general}) is always satisfied for the valuation action of 
every finitely generated abelian monoid consisting of  shift-equivariant, uniformly continuous group endomorphisms 
of a certain generalized group shift whose alphabet is not necessarily finite or even compact (e.g. compact Lie groups, algebraic groups, Artinian groups). 
To illustrate, 
we have the following particular instance of our main results (Theorem~\ref{t:main-shadow},  Theorem~\ref{t:main-shadow-subshif}):   

 \begin{corollary} 
 \label{c:intro}
Let $G$ be a countable monoid and let $A$ be an Artinian group (resp. an Artinian module). 
Let $\Sigma \subset A^G$ be a closed subshift which is also a subgroup (resp. a submodule) of $A^G$. 
Let  $\Gamma \subset \End(\Sigma)$ be a finitely generated abelian submonoid consisting of $G$-equivariant group (resp. module) endomorphisms   $\Sigma \to \Sigma$ 
that are uniformly continuous with respect to the uniform prodiscrete topology. 
Then the  valuation action of $\Gamma$ on $\Sigma$ satisfies the shadowing property. 
 \end{corollary}
 
 For example, let $A$ be a finite dimensional vector space over a field $K$. 
 Let $r \in \N$ then there is a natural shift action of the countable monoid $\N^r$ on $A^{\N^r}$. 
 Let $\Sigma \subset A^{\N^r}$ be an $\N^r$-invariant $K$-linear subspace which is closed with respect to the prodiscrete topology. 
 Suppose that $\tau  \colon \Sigma \to \Sigma$
 is a cellular automaton that is also a  $K$-linear map. 
 \par 
 Then the following shadowing property is     satisfied. Let $(E_n)_{n \geq  0}$ be an arbitrary increasing sequence of finite subsets of $\N^r$ such that $\N^r = \bigcup_{n \geq 0} E_n$ and $E_0 = \varnothing$. Such a sequence defines a \emph{Hamming metric} $d$ on $A^{\N^r}$: 
 \[
 d(x, y) \coloneqq 2^{-n(x,y)}, \quad   n(x,y) 
 \coloneqq \sup \{ n \geq 0 \colon x(g) = y(g) \mbox{ for every } g \in E_n \}.  
 \] 
 \par 
 Then for every $\varepsilon >0$, there exists $\delta >0$ such that 
 for every $\delta$-pseudo-orbit 
 $(x_n)_{n \geq 0}$ in the dynamical system $(\tau, \Sigma, d\vert_\Sigma)$, that is, $d(\tau(x_{n}), x_{n+1}) <  \delta$ for all $n \geq 0$,   we can always find  $x \in \Sigma$ such that 
$d(\tau^n(x), x_n) < \varepsilon$ for all $n \geq 0$. 
 
 \par 
A simple counter-example shows that we can not eliminate the finite generation hypothesis on the monoid $\Gamma$ in Corollary~\ref{c:intro} (cf.~Example~\ref{ex:counter}).   
\par 
 The paper is organized as follows. In Section \ref{s:pre}, we fix the notations and recall the definition of subshifts of finite type and cellular automata over monoids. 
 In Section~\ref{s:shadow}, we recall and extend the notion of pseudo-orbit and the shadowing property in~\cite{chung-shadow} to our context of the action of a finitely generated monoid $\Gamma$ on a subshift $\Sigma$. Notably, the shadowing property is independent of the choice of a finitely generating set $S$ of $\Gamma$ and of a \emph{standard Hamming metric} on $\Sigma$ (cf.~\eqref{d:hamming}). We then study in Section~\ref{s:lipschitz} the Lipschitz continuity of cellular automata with respect to a standard Hamming metric induced by a certain \emph{admissible exhaustion} of a monoid universe $G$ (Definition~\ref{d:admissible-sequence}).  In Section~\ref{s:admissible-subshift}, we extend the notion of \emph{admissible group subshift} introduced in~\cite{phung-2020} to the context of monoid universes. In particular, we show that admissible group subshifts over $\N^r$ are subshifts of finite type (cf.~Theorem~\ref{t:main-N-r-sft}). 
 This finiteness result allows us to apply a generalization of Kurka's construction of  column factorizations~\cite{kurka-97} as well as the notion of canonical factor of Blanchard-Maass~\cite{blanchard-maass-97} to establish the main results (Theorem~\ref{t:main-shadow} and Theorem~\ref{t:main-shadow-subshif}) from which we deduce Corollary~\ref{c:intro}.

 \section{Preliminaries}
 \label{s:pre}
 
 \subsection{Notation}
The set of non negative integers is denoted 
by $\N$.  
 Let $A, B$ be sets and let $C \subset B$. Then $A^B$ denotes the set of   maps from $B$ into $A$. 
If $x \in A^B$, the restriction $x\vert_C \in A^C$ is given by $x\vert_C(c) = x(c)$ for all $c \in C$.
If $X \subset A^B$, then $X_C \coloneqq \{ x\vert_C \colon x \in X \} $ is called the restriction of $X$ to $C$. For subsets $E, F$ of a monoid $G$, we denote $EF \coloneqq \{ xy\colon x \in E, y \in F \} \subset G$. 
 \subsection{Subshifts and subshifts of finite type} 
Let $G$ be a monoid, called the \emph{universe}, and let $A$ be a set, called the \emph{alphabet}. 
The \emph{right shift action} of $G$ on $A^G$ is defined  by 
$(g, x) \mapsto g\star x$ where $(g \star x)(h) \coloneqq x(hg)$ for every $g \in G$ and $x \in A^G$.   
A subset $\Sigma \subset A^G$ is \emph{$G$-invariant} if $g\star x \in \Sigma$ for all $g \in G$ and $x \in \Sigma$. 
In this case,  $\Sigma$ is called a \emph{subshift} of $A^G$. 
We do not require $\Sigma$ to be closed in $A^G$. 
\par 
Given subsets $D \subset G$ and $P \subset A^D$, 
consider the following subshift of $A^G$: 
\begin{align}
\label{e:sft} 
\Sigma(A^G; D,P)  \coloneqq  \{ x \in A^G \colon (g\star x)\vert_{D} \in  P \text{ for all } g \in G\}. 
\end{align} 
\par 
Such a set $D$ is called a
\emph{defining window} of $\Sigma(A^G; D,P)$. 
The subshift $ \Sigma(A^G; D,P)$ is clearly closed in $A^G$ with respect to the prodiscrete topology. 
If $D$ is finite,  $ \Sigma(A^G; D,P)$ is called the 
\emph{subshift of finite type} of $A^G$ associated with $D$ and $P$. 

 \par 
 
Now suppose that $H$ is a submonoid of the monoid $G$. 
Let $E \subset G$ be a subset  
such that $Hk_1\neq Hk_2$ for all distinct $k_1, k_2 \in E$.  
Denote $B \coloneqq A^E$. 
Then for every subset $F \subset H$, 
we have a  canonical bijection $B^F = A^{F   E}$ where 
every  $x \in B^F$ is mapped to an element $y \in A^{F  E}$ given  
by $y(h k) \coloneqq (x(h))(k)$ for every $h \in F$ and $k \in E$. 
\par 
We have the following elementary observation: 
 \begin{lemma}
\label{l:restriction-sft} 
With the above notations and hypotheses, let $D \subset H$ and let $P \subset A^{DE}=B^D$ be subsets. 
Then we have 
\[
\Sigma(A^G; HE, \Sigma(B^H; D, P)) = \Sigma(A^G; DE, P). 
\]
\end{lemma}

\begin{proof} 
(See also \cite[Lemma~2.2]{phung-2020})
Let us denote $\Sigma \coloneqq \Sigma(A^G; H E, \Sigma(B^H; D, P))$.  
\par 
Let $x \in \Sigma$  and let $g \in G$. 
Then  we have $(g \star x) \vert_{H E} \in \Sigma(B^H; D, P)$. 
Since $D   E \subset H   E$, 
we deduce from the canonical bijection $A^{DE} = B^D$ and the definition of $\Sigma(B^H; D, P)$ 
that $  (g \star x) \vert_{D   E} \in  P$. 
Thus, we find that $\Sigma \subset \Sigma(A^G; DE, P)$. 
\par 
Conversely, let 
$x \in \Sigma(A^G; DE, P)$ and let $g \in G$. 
Then it follows that 
$(g \star x)\vert_{DE} \in P$. 
Since $D  E \subset H E$, 
we have  
$((g \star x)\vert_{H E})\vert_{D E} \in P$ so that 
$(g \star x)\vert_{H E} \in \Sigma(B^H; D, P)$. 
Thus,  
$x \in \Sigma$ and $\Sigma(A^G; DE, P) \subset \Sigma$. 
\end{proof}

\subsection{Cellular automata over monoids}

Following the   work of John von Neumann \cite{neumann-book}, 
 cellular automata over monoids are defined as follows (cf.~\cite{csc-monoid-surj}). 
Given (finite or infinite) sets $A, B$ and a monoid $G$, a map $\tau \colon A^G  \to B^G$ is a \emph{cellular automaton}  
 if there exist a finite subset $M \subset G$ called \emph{memory set} 
and a map $\mu \colon  A^M \to B$ called \emph{local defining map} such that 
\begin{equation} 
\label{e;local-property}
\tau(x)(g) = \mu( (g\star x)\vert_M)  \quad  \text{for all } x \in A^G \text{ and } g \in G.
\end{equation}
\par 
Equivalently, a map $\tau \colon A^G \to B^G$ is a cellular automaton if and only if 
it is $G$-equivariant and uniformly continuous with respect to the prodiscrete uniform structure 
(cf.~\cite[Theorem~4.6]{csc-monoid-surj}).  
Remark   the slight difference  
with the definition of cellular automata over groups (see, e.g.~\cite{book}).

\section{Shadowing property and actions on shift spaces}
\label{s:shadow}
  
\begin{definition}
\label{d:pseudo-orbit} 
Let $X$ be a set. 
Let $S$ be a finitely generating set of a monoid $\Gamma$. 
Let $T$ be an action of $\Gamma$ on $X$  
and let $d$ be a metric on $X$. 
\begin{enumerate} [\rm (i)]
\item 
for $\delta >0$, a sequence  $(x_\tau)_{\tau\in \Gamma}$ in $X$ is called an \textit{$(S,d,\delta)$-pseudo-orbit} of $T$  if $d(T(\sigma,x_\tau), x_{\sigma \tau}) < \delta$ for all $\sigma \in S$ and $\tau\in \Gamma$.
 \item 
the action $T$ has the \textit{$(S,d)$-shadowing property}   
if for every $\varepsilon>0$, there exists $\delta>0$ such that every $(S,d,\delta)$-pseudo-orbit  $\{x_\tau\}_{\tau\in \Gamma}$ 
of $T$ is \emph{$\varepsilon$-shadowed} by some point $x$ of $X$, i.e., 
$d(T(\tau,x), x_\tau) < \varepsilon$ for all $\tau\in \Gamma$.
\end{enumerate} 
\end{definition}
\par 
Note that if $X$ is not a compact space, 
the shadowing property of $T$ might depend on the choice of the metric on $X$.  
However, the shadowing property becomes an intrinsic property  
when $X$ is a shift space over a countable monoid universe and with an arbitrary alphabet as we will see below.  
\par 
Let $G$ be a countable monoid. 
Let $(E_n)_{n \geq 0}$ be an \emph{exhaustion} of $G$, i.e., 
$E_n \subset E_{n+1}$ for all $n \geq 0$ and $G= \bigcup_{n \geq 1} E_n$, 
such that $E_0 = \varnothing$ and $E_n$ is a finite subset of $G$ for every $n \geq 0$.  
\par 
Let $A$ be a set. 
We define the \emph{standard Hamming metric} $d$ on $A^G$ associated with the exhaustion $(E_n)_{n \geq 0}$ 
by setting 
\begin{equation}
\label{d:hamming}
d(x,y) \coloneqq 2^{-n}, \mbox{ where }  n \coloneqq \sup \{ k \in \N \colon x\vert_{E_k} =y\vert_{E_k} \}. 
\end{equation}
\par 
It is immediate that a sequence $(x_n)_{n \geq 0}$ of elements of $A^G$ converges 
to some $y \in A^G$ if and only if $ \lim_{n \to \infty} d(x_n, y) = 0$. 

\begin{lemma} 
\label{l:intrinsic-def-shadow}
Let $S$  be a finitely generating set of a monoid $\Gamma$. 
Let $A$ be a set and let $G$ be countable monoid. 
Let $X \subset A^G$ be a subshift.  
Let $d$ and $d'$ be two standard Hamming metrics on $A^G$. 
Then an action $T$ of $\Gamma$ on $X$ has the $(S,d\vert_X)$-shadowing property 
if and only if it has the $(S,d'\vert_X)$-shadowing property.    
\end{lemma}

\begin{proof}
By hypotheses, 
$d$ and $d'$ are standard Hamming metrics on $A^G$ 
associated with some exhaustions  $(E_n)_{n \geq 0}$ and $(E'_n)_{n \geq 0}$ 
of finite subsets of $G$ respectively. 
Note that $(E_n)_{n \geq 0}$ and $(E'_n)_{n \geq 0}$  are cofinal, namely, 
each member of $(E_n)_{n \geq 0}$ is contained in a member of $(E'_n)_{n \geq 0}$ and vice versa.  
\par 
Let $T$ be an action of $\Gamma$ on $X$ and 
suppose that $T$ has the $(S,d\vert_X)$-shadowing property.  
\par 
Let $m'_0 \in \N$ and choose $m_0 \in \N$ such that $E'_{m'_0} \subset E_{m_0}$.  
Then there exists $n_0 \in \N$  
such that every $(S,d, 2^{-n_0})$-pseudo-orbit $\{x_\tau\}_{\tau\in \Gamma}$ of $T$ 
is $2^{-m_0}$-shadowed by a point $x \in X$ 
with respect to the metric $d\vert_X$. 
\par  
Choose $n'_0 \in \N$ such that $E_{n_0} \subset E'_{n'_0}$.  
Then every $(S,d'\vert_X,2^{-n'_0})$-pseudo-orbit of $T$  
is also a $(S, d\vert_X, 2^{-n_0})$-pseudo-orbit of $T$.  
By the choice of $n_0$ and $m_0$, such a pseudo-orbit is $2^{-m_0}$-shadowed with respect to the metric $d\vert_X$ 
by some point $x \in X$ and thus is $2^{-m'_0}$-shadowed  with respect to the metric $d'\vert_X$ by the same point $x$. 
\par 
Therefore, $T$ also has the $(S,d'\vert_X)$shadowing property.  
By exchanging the roles of $d$ and $d'$, the conclusion follows. 
\end{proof}

The above lemma leads us to the following intrinsic notion of shadowing property for actions of a finitely generated monoid on a subshift. 

\begin{definition}
\label{d:shadow-general} 
Let $A$ be a set and let $G$ be countable monoid. 
Let $\Sigma \subset A^G$ be a subshift and let $\Gamma$ be a finitely generated monoid. 
An action $T$ of $\Gamma$ on $\Sigma$ is said to have the \textit{shadowing property} 
if $T$ has the $(S,d\vert_\Sigma)$-shadowing property for 
 every finitely generating set $S$ of $\Gamma$ 
and every standard Hamming metric $d$ on $A^G$.     
\end{definition}

\section{Lipschitz continuity of cellular automata} 
\label{s:lipschitz}

We introduce the following notion of admissible exhaustive sequences of a countable monoid 
that will be used in the proof of Theorem~\ref{t:main-shadow}. 
\begin{definition} 
\label{d:admissible-sequence}
Let $G$ be a countable monoid. 
 A sequence $(E_n)_{n \geq 0}$ 
of finite subsets of $G$ is called an \emph{admissible exhaustion} of $G$ 
if it satisfies the following conditions: 
\begin{enumerate} [\rm (1)]
\item 
$E_0= \varnothing$, $1_G \in E_1$; 
\item 
$E_n^2  \subset E_{n+1}$ for every $n \geq 0$;  
\item
$G= \bigcup_{n \geq 0} E_n$; 
\end{enumerate} 
\end{definition}
\par 
Suppose that $(E_n)_{n \geq 0}$ is such an admissible exhaustion of $G$. 
Then  it follows from (1), (2) and an immediate induction that for every $n \geq 1$, we have 
$1_G \in E_n$ and thus by (2), we deduce that 
\begin{enumerate} [\rm (4)] 
\item  
$E_n \subset E_{n+1}$ for every integer $n \geq 0$; 
\end{enumerate} 
\par 
Moreover, we have the following property: 
\begin{enumerate} [\rm (5)] 
\item  
for every finite subset $M \subset G$, there exists an integer $n_0 \geq 1$ such that 
$E_n M \subset E_{n+1}$ for every $n \geq n_0$;  
\end{enumerate} 
\par 
Indeed, let $M$ be a finite subset of $G$ then 
it follows from (3) and (4) that there exists an integer $n_0 \geq 1$ such that 
$M \subset E_{n_0}$. Therefore, we infer from (2) and (4) that for every $n \geq n_0$, we have 
$E_nM \subset E_n E_{n_0} \subset E_n^2 \subset E_{n+1}$. 
\par 
We remark also that by an easy inductive construction, 
every countable monoid admits an admissible exhaustion. 
\par 
Our main motivation to introduce admissible exhaustions is that they satisfy the following useful property. 
\begin{lemma}  
\label{l:lipschitz-tau}
Let $G$ be a countable monoid and let $A$ be a set.  
Let $d$ be a standard Hamming metric of $A^G$ associated with an admissible exhaustion of $G$. 
Let $\tau \colon A^G \to A^G$ be a cellular automaton. 
Then  there exists a constant $C >0$ such that 
$\tau$ is a $C$-Lipschitz map on the metric space $(A^G, d)$. 
\end{lemma}

\begin{proof}
By hypothesis, $d$ is a standard Hamming metric of $A^G$ associated with an admissible exhaustion 
$(E_n)_{n \geq 0}$ of $G$. 
Let $M \subset G$ be a memory set of $\tau$. Then 
$M$ is finite and  by the property (5) above for $(E_n)_{n \geq 0}$, we deduce that 
there exists an integer $n_0 \geq 1$ such that $E_n M \subset E_{n+1}$ for every $n \geq n_0$. 
\par 
Let $x, y \in A^G$. We will distinguish two cases depending 
on whether $d(x,y) \leq 2^{- n_0 -1 }$ or not. 
\par 
Suppose first that $d(x,y) = 2^{-n} \leq 2^{- n_0 -1}$ for some integer $n \geq n_0+1$. 
Then $E_{n-1} M \subset E_n$ by the choice of $n_0$. 
Moreover, by definition of $d$, we have  $x\vert_{E_n}= y\vert_{E_n}$. 
Since $M$ is a memory set of $\tau$, we deduce that 
\[
(\tau(x))\vert_{E_{n-1}} = (\tau(y))\vert_{E_{n-1}} 
\] 
and it follows immediately that $d(\tau(x), \tau(y)) \leq 2^{-(n-1)} = 2 d(x,y)$. 
\par 
Now suppose that $d(x,y) > 2^{- n_0 -1 }$. Then by definition of $d$, 
we deduce that  $d(x,y) \geq 2^{-n_0}$ and thus 
$d(\tau(x), \tau(y) \leq 1 \leq 2^{n_0} d(x,y)$. 
\par   
To summarize, we have shown that for every $x, y \in A^G$, we have 
\[
d(\tau(x), \tau(y)) \leq 2^{n_0} d(x,y)
\]
which implies that $\tau$ is $2^{n_0}$-Lipschitz. 
The proof is thus completed. 
\end{proof}

\section{Admissible group subshifts} 
\label{s:admissible-subshift} 
In this section, we recall and formulate direct extensions to the case of monoid universes the general notion of admissible group subshifts introduced in \cite{phung-2020} 
as well as their basic properties. 
 
 \subsection{Admissible Artinian group structures}  
\begin{definition} [cf. \cite{phung-2020}]
\label{d:general-artinian-structure}
Let $A$ be a group. Suppose that for every integer $n\geq 1$, 
$\HH_n$ is a collection of subgroups of $A^n$ with the following properties: 
\begin{enumerate} 
\item 
$ \{1_A\}, A \in \HH_1$; 
\item 
for every $m \geq n \geq 1$ and for every  projection $\pi \colon A^m \to A^n$ induced by an injection 
$\{1, \cdots, n\} \to \{ 1, \cdots, m\}$, we have $\pi(H_m) \in \HH_n$ and $\pi^{-1}(H_n) \in \HH_m$ 
for every $H_m \in \HH_m$ and $H_n \in \HH_n$;   
\item 
for every $n \geq 1$ and $H, K\in \HH_n$, we have $H \cap K \in \HH_n$; 
\item 
for every $n \geq 1$, every descending sequence $(H_k)_{k \geq 0}$ of subgroups of $A^n$, where  
 $H_k \in \HH_n$ for every $k \geq 0$, eventually stabilizes. 
\end{enumerate} 
\par 
Then for $\mathcal{H}= (\mathcal{H}_n)_{n \geq 1}$, we say that  $(A, \HH)$, or simply $A$, 
is an \emph{admissible Artinian group structure}. 
For every $n \geq 1$, elements of $\HH_n$ are called \emph{admissible subgroups} of $A^n$ 
(with respect to the structure $\HH$). 
\par 
If $E$ is a finite set,  then $A^E$ admits an admissible Artinian structure induced by that of $A^{\{1, \cdots, |E|\}}$ via an arbitrary bijection 
$\{1, \cdots, |E|\} \to E$.    
\end{definition}

\begin{example} 
\label{example:canonical-module-admissible} 
(cf.~\cite[Examples~9.5,~9.7]{phung-2020}) 
An algebraic group $V$ over an algebraically closed field, 
resp. a compact Lie group $W$, resp. an Artinian (left or right) 
module $M$ over a ring $R$, 
admits a canonical admissible Artinian structure given by all algebraic subgroups of $V^n$, 
resp. by all closed subgroups of $W^n$, resp. 
by all $R$-submodules of $M^n$, for every $n \geq 1$  
\end{example}

\begin{example} 
\label{example:canonical-group-admissible} 
(cf.~\cite[Example~9.6]{phung-2020}) 
A group $\Gamma$ is \emph{Artinian} if every descending sequence of subgroups of $\Gamma$ eventually stabilizes.  
In this case, $\Gamma$ admits a canonical admissible Artinian structure given by all subgroups of $\Gamma^n$ for every $n\geq 1$.  
\par 
Finite groups are Artinian but not all Artinian groups are finite. For instance, given a prime number $p$, 
the subgroup $\mu_{p^\infty} \coloneqq \{ z \in \C^* \colon \exists\, n \geq 0, \, z^{p^n} = 1  \}$ 
of the multiplicative group $(\C^*, \times)$ is Artinian. 
See also \cite{schlette-artinian} for various characterizations of 
Artinian, virtually abelian groups (i.e., groups admitting a finite index abelian subgroup).   

\end{example}

\begin{definition}  [cf. \cite{phung-2020}]
\label{d:admissible-homomorphism} 
Let $A$ be a an admissible Artinian group structure. 
Let $m, n \in \N$. 
A homomorphism of abstract groups $\varphi \colon A^m \to A^n $ 
is said to be \emph{admissible} if the graph $\Gamma_\varphi \coloneqq \{ (x, \varphi(x)) \colon x \in A^m \}$ 
is an admissible subgroup of $A^{m+n}$.  
\end{definition}

Consequently, if $\varphi \colon A^m \to A^n $ is an admissible homomorphism then 
for every admissible subgroups $P \subset A^m$ and $Q \subset A^n$, the groups   
$\varphi(P)$ and $\varphi^{-1} (Q)$ are respectively admissible subgroups of $A^n$ and $A^m$. 
Indeed, let $\pi_m \colon A^m  \times A^n \to A^m$ and $\pi_n \colon A^m \times A^n \to A^n$ 
be respectively the first and the second projections. Then 
it suffices to write $\varphi(P) = \pi_n (\pi_m^{-1}(P) \cap \Gamma_{\varphi})$ 
and $\varphi^{-1} (Q) = \pi_m ( \pi_n^{-1}(Q) \cap \Gamma_\varphi)$ and apply the 
properties (2) and (3) in Definition~\ref{d:general-artinian-structure}. 
 
\begin{remark} 
\label{r:ad-ca}
Homomorphisms of algebraic groups, resp. of compact Lie groups, 
resp. of Artinian groups, and 
morphisms of $R$-modules are all admissible with the canonical admissible Artinian 
structures of algebraic groups, resp. of compact Lie groups, resp. of  Artinian groups, and of $R$-modules respectively.  
\end{remark}

  \begin{lemma}
  \label{l:ad-mor}
Let $A$ be an admissible Artinian group structure. 
Let $m , n \geq 1$  and let $E$ be a finite set. 
Let $\varphi_\alpha \colon A^m \to A^n$ be an admissible homomorphism for every $\alpha \in E$. 
Then $\varphi_E \coloneqq (\varphi_\alpha)_{\alpha \in E} \colon A^m  \to (A^n)^E$, 
$\varphi_E(x) \coloneqq (\varphi_\alpha(x))_{\alpha \in E}$ for  all $x \in A^m$, 
is also an admissible homomorphism.      
\end{lemma} 
  
\begin{proof} 
We need to show that the graph $\Gamma \coloneqq \{ (x, \varphi_E(x)) \colon x \in A^m\}$  
is an admissible subgroup of $A^m \times (A^n)^E \simeq A^{m+ n|E|}$.  
\par 
For every $\alpha \in E$,  
let  $\pi_\alpha \colon A^m \times (A^n)^E \to A^m \times (A^n)^{\{\alpha \} }$ be the canonical projection 
and let $\Gamma_\alpha \coloneqq  \{ (x, \varphi_\alpha(x)) \colon x \in A^m\} \subset A^m \times A^n$ 
be the graph of $\varphi_\alpha$. 
Remark that $\pi_\alpha$ is an admissible homomorphism for every $\alpha \in E$. 
\par 
Then we find that 
\[
\Gamma = \bigcap_{\alpha \in E} \pi_\alpha^{-1} (\Gamma_\alpha)    
\]
which is clearly an admissible subgroup of $A^m \times (A^n)^E$ as an intersection 
of admissible subgroups of $A^m \times (A^n)^E$. The proof is completed. 
\end{proof}

A similar argument using graphs shows easily the following: 

\begin{lemma}
Let $G$ be a monoid and let $A$ be an admissible Artinian group structure. 
Let $F \subset G$ be a finite subset and let $m, n \geq 1$ be integers. 
Let $\tau \colon (A^m)^G \to (A^n)^G$ be an admissible group cellular automaton 
with a memory set $M \subset G$.  
Then the induced map  
$\tau_{F}^{+} \colon (A^m)^{FM} \to (A^n)^F$,  
$\tau_F^{+}(c) \coloneqq \tau(x)\vert_F$ for all $c \in (A^m)^{FM}$ and   $x \in (A^m)^G$ such that $x \vert_{FM}=c$, 
is an admissible homomorphism. 
\end{lemma}
 
 \begin{proof} 
 The proof of the lemma is the same, \emph{mutatis mutandis}, as the proof of  \cite[Lemma~9.20]{phung-2020} for the case of group universes. 
 \end{proof}

 \subsection{Admissible group subshifts}

By analogy with the classical notion of \emph{group shifts} with finite group alphabets (see, e.g.,~\cite{fiorenzi-periodic}), 
we have the following notion of admissible group subshifts whose alphabets are admissible Artinian group structures introduced in \cite{phung-2020}.   

\begin{definition}
\label{d:admissible-subgroup-shift}
Let $G$ be a monoid and let $A$ be an admissible Artinian group structure. 
A subshift $\Sigma \subset A^G$ is called an \emph{admissible group subshift}  
if it is closed in $A^G$ with respect to the prodiscrete topology and 
if $\Sigma_E$ is an admissible subgroup of $A^E$
for every finite subset $E \subset G$.        
\end{definition}

\par 
\begin{example} 
\label{example:canonical-admissible-for-intro} 
Let $G$ be a monoid and let $A$ be an Artinian group (resp. an Artinian module over a ring  $R$). 
 Let $\Sigma$ be a closed subshift of $A^G$ with respect to the prodiscrete topology which  
is also an abstract subgroup (resp. an $R$-submodule). 
Then  
$\Sigma$ is   an admissible group subshift of $A^G$ with respect to 
 the canonical admissible Artinian group  structure on $A$ (cf.~Example~\ref{example:canonical-group-admissible}, Example~\ref{example:canonical-module-admissible}). 
  
 \end{example}

\subsection{Admissible group cellular automata} 

\begin{definition} 
Let $G$ be a monoid and let $A$ be an admissible Artinian group structure. 
Let $n_1, n_2 \geq 1$ and let $\Sigma_1 \subset (A^{n_1})^G$,   
$\Sigma_2 \subset (A^{n_2})^G$ 
be admissible group subshifts. 
A map $\tau \colon \Sigma_1 \to \Sigma_2$ is called an \emph{admissible group cellular automaton}  
if $\tau$ extends to a cellular automaton $(A^{n_1})^G \to (A^{n_2})^G$ admitting a memory set $M \subset G$ such that the associated local defining map 
$\mu \colon (A^{n_1})^M \to A^{n_2}$ is an admissible homomorphism 
(cf.~Definition~\ref{d:admissible-homomorphism} and Remark~\ref{r:ad-ca}). 
\par
When $n_1=n_2$ and $\Sigma_1 = \Sigma_2= \Sigma$, 
we denote by $\End_{G\text{-grp}} (\Sigma)$ 
the set of all such admissible group cellular automata.  
\end{definition} 
\par 

The following result provides a large class of admissible group subshifts and shows that 
admissible group subshifts are stable under taking images of admissible group cellular automata. 

\begin{theorem} 
\label{t:artinian-general-sft} 
Let $G$ be a countable monoid  
and let $A$ be an admissible Artinian group structure. 
Let $m, n \geq 1$ be integers.  
The following hold: 
\begin{enumerate} [\rm (i)]
\item 
if $D \subset G$ is a finite subset and $P \subset A^D$ is an admissible subgroup,  
then $\Sigma(A^G; D, P)$ is an admissible group subshift of $A^G$. 
\item 
if $\tau \colon (A^m)^G \to (A^n)^G$ is an admissible group cellular automaton and 
  $\Sigma \subset (A^m)^G$, $\Lambda \subset (A^n)^G$  
are admissible group subshifts, then   
$\tau(\Sigma)$, $\tau^{-1}(\Lambda)$ are  respectively 
admissible group subshifts of $(A^n)^G$, $(A^m)^G$. 
\end{enumerate} 
\end{theorem}

\begin{proof} 
It is a direct extension of \cite[Theorem~9.16]{phung-2020} and \cite[Theorem~9.21]{phung-2020} 
where the case when $G$ is a countable group is proved.  
The modifications to case when $G$ is a countable monoid are straightforward and we omit the details.  
\end{proof}

\section{Monoid of admissible Markov type}
\label{s:markov-type}

By analogy with the definition of groups of Markov type given in  \cite[Definition~4.1]{schmidt-book},   
we introduce the class of \emph{monoids of admissible Markov type} as follows. 

\begin{definition}
\label{d:markov-type-alg}
A countable monoid $G$ said to be of \emph{admissible Markov type} 
if for every admissible Artinian group structure $A$, 
every admissible group subshift $\Sigma$ of $A^G$ is of finite type.  
\end{definition}
\par 
By the descending chain condition of admissible Artinian group structures, 
 every finite monoid is a monoid of admissible Markov type. 
\par
We have the following result which is sufficient for our purpose: 

\begin{theorem} 
\label{t:descending-ad-sft}
Let $G$ be a countable monoid which is of admissible Markov type. 
Let $V$ be an admissible Artinian group structure.  
Then every descending sequence of admissible group subshifts of $V^G$ 
eventually stabilizes. 
\end{theorem} 

\begin{proof} 
This is a direct extension of \cite[Proposition~9.17]{phung-2020}. 
Actually, the theorem generalizes the part (a)$\implies$(c) of \cite[Theorem~4.3]{phung-2020} and \cite[Theorem~10.1]{cscp-2020} 
which are stated for algebraic group subshifts over countable group universes. 
It is a straightforward verification that their proofs also work for countable monoid universes 
and admissible group subshifts. 
\end{proof}
 
\par  
Our main  goal in this section is to give a proof of the following  result which is analogous to \cite[Theorem~1.6]{phung-2020}, 
\cite[Theorem~4.2]{schmidt-book} and \cite[Propositoin~3.2]{kitchens-schmidt}.

\begin{theorem}
\label{t:main-N-r-sft} 
Let $F$ be a finite monoid and let $r \in \N$. 
Then the monoid $\N^r \times F$ is of admissible Markov type. 
\end{theorem} 

We shall need the  following key technical lemmata 
on inverse systems 
of admissible Artinian group structures.

  \begin{lemma}
\label{l:lemma-artinian-1} 
Let $\Gamma$ be an admissible Artinian group structure. 
Let $(X_n)_{n \geq 0}$ be a descending sequence of left translates of admissible subgroups of $\Gamma$. 
Then the sequence $(X_n)_{n \geq 0}$ eventually stabilizes. 
\end{lemma}
 
 \begin{proof} 
 See \cite[Lemma~9.14]{phung-2020}. 
 \end{proof}

\begin{lemma}
\label{l:inverse-limit-ad-grp} 
Let $(\Gamma_i, \varphi_{ij})_{i,j \in I}$ 
be an inverse system indexed by a countable directed set $I$, 
where every $\Gamma_i$ is an admissible Artinian group structure and the transition maps  
$\varphi_{i j} \colon \Gamma_j \to \Gamma_i$ are admissible homomorphisms for all $i \prec j$. 
Suppose that $X_i$, for every $i \in I$, is  
a left translate of an admissible subgroup of $\Gamma_i$ and that $\varphi_{ij}(X_j)\subset X_i$ 
for all  $i\prec j$ in $I$. 
Then the induced inverse subsystem $(X_i)_{i \in I}$ satisfies 
$\varprojlim_{i \in I} X_i \neq \varnothing$.  
\end{lemma}

\begin{proof} 
See \cite[Lemma~9.15]{phung-2020}. 
\end{proof}

\subsection{The case of infinite cyclic extension} 
\label{s:main-infinite-cyclic}
We will now prove an extension  
of \cite[Theorem~7.2]{phung-2020} and \cite[Lemma~4.4]{schmidt-book}. 
The proof follows closely the steps of the proof of \cite[Theorem~7.2]{phung-2020} with some minor but not straightforward modifications.  
Hence, for the convenience of the readers and for sake of completeness, we include the details below.   

\begin{theorem} 
\label{t:ad-group-Z-d} 
Let  $H$ be a countable monoid of admissible Markov type. 
Then $ H \times \N$ is also a monoid of admissible Markov type. 
\end{theorem} 

\begin{proof} 
Let $G \coloneqq H \times \N$. 
Let $V$ be an admissible Artinian group structure. 
Let $\Sigma \subset V^G$ be an admissible group subshift. 
We must show that $\Sigma$ is an admissible group subshift of finite type 
of  $V^G$. 
 \par 
Let $0_H = 0_G$ be the neutral element of the monoids $H$ and $G$ whose laws are denoted 
additively just as $\N$. The group laws on $V$ is written multiplicatively 
and $\varepsilon$ stands for the neutral element of $V$.  
 \par 
Since $H$ is countable, we can find an increasing sequence $(F_n)_{n \geq 1}$ 
of finite subsets of $H$ such that $0_H \in F_1$ and $H = \bigcup_{n \geq 1} F_n$. 
\par 
For every integer $n \geq 1$, let us denote $G_n \coloneqq  \{0, \dots, n \} \subset \N$.  
We define   
\begin{align}
\label{e:def-x-n}
X_n  \coloneqq \{x\vert_{H \times \{n\} } \colon x \in \Sigma , \, x\vert_{H \times G_{n-1}} = \varepsilon^{H \times G_{n-1}}\} 
\subset   V^{H \times \{n\}}.  
\end{align}
For every $n \in \N$, there is a canonical bijection 
$\Phi_{H, n} \colon V^{H \times \{ n \}}  \to V^H$ induced by $h \mapsto (h,n)$. 
Since $\Sigma$ is $G$-invariant, we deduce easily that $\Phi_{H, n+1} (X_{n+1} ) \subset \Phi_{H, n} (X_n)$ and 
$\Phi_{H, n}  (X_n)$ is $H$-invariant for all $n \geq 1$. 
\par 
 
\begin{lemma} 
\label{claim:1}
Let $E \subset H$ be a finite subset. Then $(X_n)_{E \times \{n\}}$  
is an admissible subgroup of $V^{E \times \{n \}}$ for every $n \geq 1$. 
\end{lemma} 
\begin{proof} 
Fix $n \geq 1$. Since $E$ is finite, there exists an integer $k_0 \geq n$ such that $E \subset F_{k_0}$.  
Consider the following subset $Y \subset \Sigma_{E \times \{ n \}}$ defined by:  
\begin{align}
\label{e:y-claim-1}
Y \coloneqq \bigcap_{k \geq k_0}  Y_k , \quad \mbox{where} \quad 
Y_k \coloneqq  \{ x\vert_{E \times \{n \} } \colon x \in \Sigma, \, x\vert_{ F_k \times G_{n-1}} = \varepsilon^{F_k \times G_{n-1}}  \}.  
\end{align}
\par 
Let $k \geq k_0$ be an integer. Since $\Sigma$ is an algebraic group subshift, 
$\Sigma_{ F_k \times G_n}$ is an algebraic subgroup of $V^{F_k \times G_n}$. 
Let $\pi^* \colon \Sigma_{F_k \times G_n} \to \Sigma_{F_k \times G_{n-1}}$ 
and $\pi_E \colon \Sigma_{F_k \times G_n} \to \Sigma_{E \times \{ n \}}$  
be the canonical homomorphisms induced respectively by the inclusions 
$F_k \times G_{n-1} \subset F_k \times G_n$ and $E \times \{n\} \subset F_k \times G_n$  
(cf.~also~\cite[Lemma~4.7]{phung-2020}).   
\par 
Then  $Y_k = \pi_E(\Ker(\pi^*))$ is an admissible subgroup of $\Sigma_{E \times \{n\} }$ 
and thus of $V^{E \times \{n\}}$.  
Thus the descending chain condition of $V^{E \times \{n\} }$ implies that 
the descending sequence of admissible subgroups $(Y_k)_{k \geq k_0}$ 
of $V^{E \times \{n\}}$ stabilizes. 
It follows that $Y$ is an admissible subgroup of $V^{E \times \{n\} }$.  
\par 
Next, we will show  that  $(X_n)_{E \times \{n\} }= Y$. 
 The inclusion $(X_n)_{E \times \{n\} } \subset Y$  
 is immediate. For the converse inclusion, let 
 $y \in Y$. We must show that there exists $x \in X_n$ such that $x\vert_{E \times \{n\} } =  y$.  
For every $k \geq k_0$,   consider the following subset of $\Sigma_{F_k \times G_k}$:   
\begin{align}
\label{e:y-k-y}
Y_k(y) \coloneqq  \{ x\vert_{F_k \times G_k} \colon x \in \Sigma, 
\, x\vert_{E \times \{n \}} =   y, \, x\vert_{F_k  \times G_{n-1}} = \varepsilon^{F_k \times G_{n-1}} \}. 
\end{align} 
\par 
Since $y \in Y = \bigcap_{k \geq k_0}  Y_k$, we can find for every $k \geq k_0$  
a configuration $x_k \in \Sigma$ such that $y= x_k \vert_{E \times \{n\} }$ and 
$x_k\vert_{F_k \times G_{n-1}} = \varepsilon^{F_k \times G_{n-1}}$. 
This shows that $x_k\vert_{F_k \times G_k} \in Y_k(y)$ for every $k \geq k_0$. 
\par 
For every   $k \geq k_0$, we have canonical homomorphisms  
$\psi_k^* \colon \Sigma_{F_k\times G_k} \to \Sigma_{F_k \times G_{n-1}}$ 
and $\phi_k \colon \Sigma_{F_k\times G_k} \to \Sigma_{E \times \{n\} }$ 
induced by the corresponding inclusions of sets in the indices.  
A direct verification shows that 
\begin{align} 
\label{e:y-k-y-translate} 
Y_k(y) =  x_k\vert_{F_k \times G_k}  ( \Ker (\psi_k^*) \cap \Ker(\phi_k)).
\end{align}
\par 
Thus $Y_k(y)$ is a translate by $ x_k\vert_{F_k \times G_k} $ of the admissible subgroup 
$\Ker (\psi_k^*) \cap \Ker(\phi_k)$ of $\Sigma_{F_k\times G_k}$ for every $k \geq k_0$.  
\par 
Remark that $n$ is fixed so that the sequence $(Y_k(y))_{k \geq k_0}$ 
forms an inverse system of nonempty sets. The  
transition maps $Y_m(y) \to Y_k(y)$, 
where $m \geq k \geq k_0$, are the restrictions of the canonical homomorphisms  
$\Sigma_{F_m\times G_m} \to \Sigma_{F_k\times G_k}$ induced by the inclusions $F_k \times G_k \subset F_m\times G_m$. 
\par 
Since $Y_m(y)$ is a translate of an admissible subgroup of $\Sigma_{F_m \times G_m}$, 
the transition maps $Y_m(y) \to Y_k(y)$ have admissible images for all $m \geq k \geq k_0$. 
\par 
 Lemma~\ref{l:inverse-limit-ad-grp} then implies that there exists 
$x \in \varprojlim_{k \geq k_0} Y_k(y)$. 
By the construction of the sets $Y_k(y)$, we have for every $k \geq k_0$ that 
\[
x\vert_{E \times \{n\} } = y, \quad \mbox {and } \quad  
x\vert_{F_k \times G_{n-1}} = \varepsilon^{F_k \times G_{n-1}}. 
\]  
\par 
Since $H = \bigcup_{k \geq k_0} F_k$, 
we deduce that 
$x\vert_{H \times G_{n-1}} = \varepsilon^{H \times G_{n-1}}$.  
\par 
Note that $ \varprojlim_{k \geq k_0} Y_k(y) \subset \varprojlim_{k \geq n} \Sigma_{\Phi(F_k\times G_k)}$ 
since $Y_k(y) \subset \Sigma_{\Phi(F_k\times _k)}$ for every $k \geq k_0$.  
By the closedness of $\Sigma$ in $V^G$ with respect to the prodiscrete topology 
and as $G = \bigcup_{k \geq n} \Phi(F_k\times G_k)$, 
we have $\varprojlim_{k \geq n} \Sigma_{\Phi(F_k\times G_k)} = \Sigma$.   
It follows that $x \in \Sigma$. 
\par 
We deduce that  $x \in X_n$ by definition of $X_n$ (cf.~\eqref{e:def-x-n})). 
Since $x\vert_{E \times \{n\}} = y$ as well, we have $Y \subset (X_n)_{E \times \{n\}}$.  
\par 
We conclude that  $(X_n)_{E \times \{n\}} = Y$ is an admissible subgroup 
of $V^{E \times \{n\}}$ and thus Lemma~\ref{claim:1} is proved. 
\end{proof}

\begin{lemma}
\label{claim:2}
For every integer $n \geq 1$, $\Phi_{H, n}(X_n)$ is an $H$-invariant closed subset of $V^H$ 
with respect to the prodiscrete topology.  
\end{lemma} 
\begin{proof}
Let us fix $n \geq 1$.  
For every $k \geq n \geq 1$, define  
 \[
X_{nk} \coloneqq \{x\vert_{F_k \times G_k} \colon x \in \Sigma , \, x\vert_{F_k \times G_{n-1}} = \varepsilon^{F_k \times G_{n-1}}\} 
\subset \Sigma_{F_k \times G_k}
\]
which  is exactly the kernel of 
 the admissible homomorphism of admissible subgroups $\Sigma_{F_k \times G_k} \to \Sigma_{F_k \times G_{n-1}}$. 
It follows that $X_{nk}$ is  an admissible subgroup  of $\Sigma_{F_k \times G_k}$ and thus of $V^{F_k \times G_k}$.  
\par
For $m \geq k \geq n$, the inclusion $F_k \times G_k \subset F_m \times G_m$ 
induces a projection $\pi_{km} \colon V^{F_m \times G_m} \subset V^{F_k\times G_k}$. 
If $x \in V^{F_m \times G_m}$ satisfies $x\vert_{F_m \times G_{n-1}} = \varepsilon^{F_m \times G_{n-1}}$ 
then clearly  $\pi_{km}(x)\vert_{  F_k \times G_{n-1}} = \varepsilon^{F_k \times G_{n-1}}$. 
Hence, the restriction of $\pi_{km}$ to $X_{nm}$ defines a homomorphism of admissible groups 
$p_{km} \colon X_{nm} \to X_{nk}$. 
\par 
We thus obtain an inverse system $(X_{nk})_{k \geq n}$ whose transition maps $p_{km}$ 
are homomorphisms of admissible groups for $m \geq k \geq n$.   
\par 
Now suppose that $z \in V^{H \times \{n\}}$ belongs to the closure of $X_n$ in $V^{H \times \{n\}}$ with respect to the 
prodiscrete topology. 
Hence, by definition of $X_n$, there exists for each $k \geq n$ 
a configuration   
$y_k \in \Sigma$ such that $ y_k \vert_{F_k \times \{n\}} = z\vert_{F_k\times \{n\}}$ 
and that $y_k \vert_{H \times G_{n-1}} = \varepsilon^{H \times G_{n-1}}$. 
\par 
For every $k \geq n$, consider the following subset of $\Sigma_{F_k\times G_k}$: 
 \begin{equation}
 \label{e:x-n-k-1}
X_{nk}(z) \coloneqq \{ x\vert_{F_k\times G_k} \colon 
 x \in \Sigma , \, x\vert_{F_k \times G_{n-1}} = \varepsilon^{F_k \times G_{n-1}}, \, 
 x\vert_{F_k \times\{n\}} = z \vert_{F_k \times \{n\}}  \}. 
\end{equation}
\par 
Observe that $y_k\vert_{ F_k\times G_k} \in X_{nk}(z)$ for every $k \geq n$. 
As in \eqref{e:y-k-y-translate}, we find that $X_{nk}(z)$ is a translate of an admissible subgroup of $\Sigma_{F_k\times G_k}$. 
\par 
Then by Lemma~\ref{l:inverse-limit-ad-grp},  
there exists $x \in \varprojlim_{k \geq n} X_{nk}(z) \subset \varprojlim_{k \geq n} \Sigma_{F_k\times G_k} = \Sigma$. 
We find that 
$x\vert_{F_k \times G_{n-1}} = \varepsilon^{F_k \times G_{n-1}}$ 
and that $x\vert_{F_k \times \{n\}} = z \vert_{F_k \times \{n\}}$ for every $k \geq n$. 
Thus, by letting $k \to \infty$, we obtain 
$x\vert_{H \times G_{n-1}} = \varepsilon^{H \times G_{n-1}}$ and 
$x\vert_{H \times \{n\}} = z$. 
Hence, $z  \in X_n$ and this proves that $\Phi_{H,n}(X_n)$ is closed in $V^H$ 
with respect to the prodiscrete topology. It is trivial that $\Phi_{H,n}(X_n)$ is an $H$-invariant subset 
of $V^H$. Lemma~\ref{claim:2} is thus proved. 
 \end{proof}

By Lemma~\ref{claim:1},  Lemma~\ref{claim:2} and the remarks  
after the definition \eqref{e:def-x-n} of $X_n$, we deduce the following:   

\begin{lemma} 
\label{l:lemma-1}
The sequence $(\Phi_{H,n}(X_n))_{n \geq 1}$ is a descending sequence of admissible group subshifts of $V^H$. 
\qed
\end{lemma}  
 Now since $H$ is of admissible Markov type, 
Lemma~\ref{l:lemma-1} and Theorem~\ref{t:descending-ad-sft} imply that  $(\Phi_{H,n}(X_n))_{n \geq 1}$ 
must stabilize and consist of admissible group subshifts of finite type of $V^H$. 
Hence, there exists $N \geq 1$ such that  
\begin{equation} 
\label{e:proof-main-X} 
\Phi_{H,n}(X_n) = \Phi_{H,n}(X_N) \eqqcolon X \quad \mbox{for every $n \geq N$}.  
\end{equation}
\par 
  For every   $n \geq 1$ and $v \in X$, consider the following subset  of $\Sigma$  
\[
L_{v, n} \coloneqq  \{ x \in \Sigma \colon x\vert_{H \times G_{n-1}} = \varepsilon^{H \times  G_{n-1}}, \, 
x\vert_{H \times \{n\}} = \Phi_{H, n}^{-1} (v) \}. 
\]
The relation \eqref{e:proof-main-X} implies that $L_{v,n}$ is nonempty for every $n \geq 1$ and $v \in X$. 
\par 
 Let $\Omega \coloneqq  H \times G_N \subset G$ and 
let $\Omega(1) \coloneqq H \times \{ 1, \cdots, N +1\}$ be the translate of $\Omega$ 
by $(0_H, 1)$. 
Consider the subshift 
$\Sigma' \coloneqq  \Sigma (V^G; \Omega, \Sigma_\Omega)$ of $V^G$ (see Definition \eqref{e:sft}). 
 It is clear that $\Sigma \subset \Sigma'$. 
We are going to prove the converse inclusion.  
\par 
 Let $y \in \Sigma'$ be a configuration.  
  Then by definition of $\Sigma'$, there exists $z_0 , z_1 \in \Sigma$ 
such that $(z_0)\vert_\Omega = y\vert_\Omega$ 
and  $(z_1)\vert_{\Omega(1)} = y\vert_{\Omega(1)}$. 
It follows that for $z = z_0(z_1)^{-1} \in \Sigma$, we have 
$z\vert_{H \times \{1, \cdots, N \}} = \varepsilon^{H \times \{ 1, \cdots, N \}}$. 
\par 
Therefore, $v \coloneqq z\vert_{H \times \{N+1\}} \in \Phi_{H,N+1}^{-1} (X)$. 
\par 
 Let $c \in L_{v,N+1}$. 
Then the configuration $x \coloneqq c^{-1}z_0 \in \Sigma$ satisfies 
\[
x\vert_{  H \times G_{N+1}} = y\vert_{ H \times G_{N+1}}. 
\]  
An immediate induction on $m \geq 1$ by a similar argument shows that 
there exists a sequence $(x_m)_{m \geq 1} \subset \Sigma$ such that 
$x_m\vert_{H \times G_m} = y\vert_{H \times G_m}$ for every $m \geq 1$. 
\par 
Remark that any given finite subset of $G$ is contained in 
some translate of the sets $H  \times G_m$ for some $m \geq 1$. 
Consequently, the above paragraph shows  that $y$ belongs to the closure of $\Sigma$ in $V^G$ 
with respect to the prodiscrete topology.  
As $\Sigma$ is closed in $V^G$, it follows that $y \in \Sigma$. 
Therefore, $\Sigma' \subset \Sigma$ and we conclude that $\Sigma = \Sigma' \subset V^G$.  
\par 
We regard $\Sigma_\Omega$ as a subshift of $U^H$ with respect to the right shift action given by the monoid $H$ 
with the alphabet 
$U \coloneqq V^{\{0_H\} \times G_N}$ which is an admissible Artinian group structure. 
\par 
As $\Sigma$ is closed in $V^G$ with respect to the prodiscrete topology, thus 
\cite[Lemma~3.1]{phung-2020} (which holds for admissible group subshifts over countable monoid universes by a similar proof) 
implies  
that $\Sigma_\Omega$ is closed in $U^H$. 
On the other hand, $(\Sigma_\Omega)_{ E \times G_N} = \Sigma_{E \times G_N}$ is an admissible subgroup of $U^E$ for every finite subset $E \subset H$. 
Hence, $\Sigma_\Omega$ is an admissible group subshift of $U^H$. 
\par 
Since $H$ is a monoid of admissible Markov type, $\Sigma_\Omega$ is an admissible group subshift of finite type of $U^H$. 
Thus,  there exists a finite subset $D \subset H$ such that 
$\Sigma_\Omega = \Sigma(U^H; D, P)$ 
where $P \coloneqq (\Sigma_\Omega)_{D \times G_N} = \Sigma_{D \times G_N}$ 
is  an admissible subgroup of $U^D=V^{D \times I_N}$.  
It follows that  
\begin{align*}
\Sigma & = \Sigma'= \Sigma(V^G; \Omega, \Sigma_\Omega) &&\\
& = \Sigma(V^G;  H \times G_N, \Sigma(U^H; D, P)) &&\\
& = \Sigma(V^G;  D \times G_N, P) &&\mbox{(by Lemma~\ref{l:restriction-sft})}.
\end{align*} 
Since $D \times G_N$ is finite and $P$ is an admissible subgroup of $V^{D \times G_N}$, 
we conclude that $\Sigma$ is an admissible group subshift of finite type of $V^G$.  
\par 
The proof  of Theorem~\ref{t:ad-group-Z-d} is complete. 
\end{proof}

\subsection{The case of extension by finite groups} 
\label{s:main-finite-cyclic} 

The following proposition is a direct application of Lemma~\ref{l:restriction-sft}.  

\begin{proposition}
\label{p:main-by-finite} 
Let  $H$ be a countable monoid of admissible Markov type 
and let $F$ be a finite monoid. 
Then $ H \times F$ is also a monoid of admissible Markov type. 
 \end{proposition}

\begin{proof}
Using Lemma~\ref{l:restriction-sft}, the proof of the proposition is identical, \emph{mutatis mutandis}, to the proof of \cite[Proposition~7.6]{phung-2020}.  
\end{proof}

\begin{proof}[Proof of Theorem~\ref{t:main-N-r-sft}]
It is a direct consequence of Theorem~\ref{t:ad-group-Z-d} and Proposition~\ref{p:main-by-finite}.  
\end{proof}

\section{Column factorizations} 
\label{s:kurka-construction} 
We generalize the   useful construction of \emph{column factorizations} of Kurka \cite{kurka-97} as follows (see also the similar notion of  \emph{canonical factor} of Blanchard-Maass \cite{blanchard-maass-97}). 
Let $G$ be a countable monoid and let $A$ be a set. 
Let $\Sigma$ be a subshift of $A^G$ and let 
$\tau_1, \cdots, \tau_r \colon \Sigma \to \Sigma$ be cellular automata (in that order).   
For every  $\alpha  = (a_1, \cdots, a_r)  \in \N^r$,  we  denote 
\begin{equation} 
\label{e:tau-alpha}
\tau_\alpha \coloneqq \tau_1^{a_1} \circ   \cdots \circ    \tau_r^{a_r}.
\end{equation}
\par 
Suppose that the cellular automata $\tau_1, \cdots, \tau_r$ 
are pairwise commuting.  Hence, we can simply write $\tau_\alpha = \tau_1^{a_1}   \cdots    \tau_r^{a_r}$ 
for the composition $\tau_1^{a_1} \circ   \cdots \circ    \tau_r^{a_r}$. 
\par 
Let $E \subset G$ be a finite subset. 
We define a map $\Psi_E \colon \Sigma \to  (\Sigma_E)^{\N^r} $ as follows.  
For every $x \in \Sigma$ and $ \alpha \in \N^r$, we set:  
\begin{equation} 
\label{e:psi-e}
\Psi_E(x) (\alpha)  \coloneqq  \left( \tau_\alpha (x) \right)\vert_E. 
\end{equation}
\par 
\begin{definition} 
With the above notations, the subset  
\begin{equation}
\label{e:column-facto}
\Lambda(\Sigma, E ;  \tau_1, \cdots, \tau_r) \coloneqq \Psi_E(\Sigma) \subset (\Sigma_E)^{\N^r}
\end{equation}
is called the \emph{column factorization} associated with $\Sigma$, $E$ and $\tau_1, \cdots, \tau_r$.  
\end{definition}

We have the following crucial property of column factorizations of an admissible group subshift.

\begin{theorem} 
\label{l:closed-kurka-sft}
Let $G$ be a countable monoid and let $A$ be an admissible Artinian group structure. 
Let  $\Sigma \subset A^G$ be an admissible   group subshift  
and let $\tau_1, \cdots, \tau_r \in \End_{G\text{-grp}} (\Sigma)$ be pairwise commuting. 
Let $E \subset G$ be a finite subset.  
Then $\Lambda(\Sigma, E ;  \tau_1, \cdots, \tau_r) \subset (\Sigma_E)^{\N^r}$  is  a subshift of finite type.  
\end{theorem}

\begin{proof}
Let us denote   
$\Lambda \coloneqq \Lambda(\Sigma, E ;  \tau_1, \cdots, \tau_r) = \Psi_E(\Sigma) \subset (\Sigma_E)^{\N^r}$ 
where 
$\Psi_E(x) (\alpha)  \coloneqq  \left( \tau_\alpha \right)\vert_E$ 
for $x \in \Sigma$ and $\alpha   \in \N^r$ (cf.~\eqref{e:psi-e}). 
\par 
We are going to show that the subshift 
$\Lambda$ is an admissible group subshift of $(\Sigma_E)^{\N^r}$. 
\par 
First, observe that $\Lambda$ is $\N^r$-invariant. Indeed, 
let $y \in \Lambda$ and let $\beta \in \N^r$. 
Then there exists $x \in \Sigma$ such that $y= \Psi_E(x)$. 
Since $\tau_1, \cdots, \tau_r$ are pairwise commuting, 
we have for every $\alpha \in \N^r$ that 
\begin{align} 
(\beta \star y)( \alpha ) & = y ( \alpha + \beta ) 
= \Psi_E( x ) ( \alpha + \beta ) \\
& = \tau_{\alpha + \beta}(x)\vert_E 
= \tau_\alpha( \tau_\beta  (x) )\vert_E \nonumber \\
& = \Psi_E(\tau_\beta(x))(\alpha) \nonumber
\end{align} 
\par 
It follows that $\beta \star y = \Psi_E(\tau_\beta(x)) \in \Lambda$ and thus 
$\Lambda$ is a subshift of $(\Sigma_E)^{\N^r}$. 
\par 
Now let $F \subset \N^r$ be a finite subset, we must show that $\Lambda_F$ is an admissible subgroup 
of $A^F$.  To see this, choose $M_F \subset G$ to be a large enough finite subset 
such that $M_F$ is memory set of every cellular automaton $\tau_\alpha$ where 
$\alpha \in F$. 
\par
Let $B \coloneqq A^F$ then $B$ inherits from $A$ an admissible Artinian group structure. 
For every $\alpha \in F$, 
let $\mu_\alpha \colon A^{M_F} \to A$ be the local defining map of $\tau_\alpha$ 
and denote by $\tilde{\tau}_\alpha \colon A^G \to A^G$ the corresponding induced cellular automaton that extends $\tau_\alpha$. 
Then $\mu_\alpha$ is an admissible homomorphism and $\tilde{\tau}_\alpha $ is an admissible group cellular automaton. 
\par  
Consider the cellular automaton $\tau_F \coloneqq \left( \tilde{\tau}_{\alpha}\right)_{\alpha \in F} \colon A^G \to B^G$ defined by: 
\[
\tau_F (x) (g) \coloneqq (\tilde{\tau}_\alpha(x)(g))_{\alpha \in F}, \quad \mbox{for all $x \in A^G$ and $g \in G$}. 
\] 
\par 
Then $\tau_F$ is  a cellular automaton 
admitting $M_F$ as a memory set whose associated local defining map 
is given by 
\begin{align} 
\label{e:mu-f-main} 
\mu_F \colon A^{M_F} \to B,\quad \mu(c) = (\mu_\alpha(c))_{\alpha \in F}, \quad c \in A^{M_F}.   
\end{align}

\par 
We infer from Lemma~\ref{l:ad-mor} that $\mu_F$ is an admissible homomorphism and thus  
$\tau_F$ is an admissible group cellular automaton.  
Thus, Theorem~\ref{t:artinian-general-sft} implies that $\tau_F(\Sigma)$ is 
an admissible group subshift of $B^G$ and thus 
$\tau_F(\Sigma) \vert_E$ is an admissible subgroup of $ B^E = (A^F)^E$. 
\par
It is clear by construction that $\Lambda_F = \tau_F(\Sigma) \vert_E$ via the canonical 
bijections 
\[ 
(A^E)^F = A^{E \times F}  = A^{F \times E} = (A^F)^E. 
\] 
\par 
It follows that $\Lambda_F$ is an admissible subgroup of $(A^E)^F$. 
\par 
Now let $z \in (\Sigma_E)^{\N^r}$ that belongs to the closure of $\Lambda$ 
in $(\Sigma_E)^{\N^r}$. 
 Let $(F_n)_{n \geq 1}$ be  an exhaustion of $\N^r$ consisting of finite subsets. 
 We can clearly suppose that the sequence $(M_{F_n})_{n \geq 1}$ also forms 
 an exhaustion of $G$.    
 \par    
 For every $n \geq 1$, we define (cf.~\eqref{e:mu-f-main}): 
 \begin{align}
 \label{e:x-n-z}
 X_{n} (z) \coloneqq \{ x \in \Sigma_{M_{F_n}} \colon  
  \mu_{F_n} (x) = z\vert_{F_n}   \} = \mu_{M_{F_n}}^{-1} (z\vert_{F_n}) \subset \Sigma_{M_{F_n}}. 
 \end{align}
 \par 
 Since $z$ belongs to the closure of $\Lambda$ in   $(\Sigma_E)^{\N^r}$, 
 it is immediate from the definition of $\Lambda$ and  \eqref{e:x-n-z} 
 that $X_{n}(z)$ is nonempty for all $n \geq 1$. 
 \par 
 Since $\mu_{M_{F_n}}$ is an admissible homomorphism, we deduce that 
 $X_n(z) $ is the left translate of an admissible subgroup of $\Sigma_{M_{F_n}}$ and $A^{M_{F_n}}$ 
 for every $n \geq 1$. 
\par 
Therefore, we obtain an inverse system $(X_n(z))_{n \geq 1}$ 
whose transition maps $X_m(z) \to X_n(z)$, where $m \geq n \geq 1$, 
are the restrictions of the canonical projections $\pi_{m,n} \colon A^{M_{F_m}} \to A^{M_{F_n}}$. 
Remark that the maps $\pi_{m,n}$ are admissible homomorphisms. 
\par  
Hence, Lemma~\ref{l:inverse-limit-ad-grp} implies that  
there exists 
\[
x \in \varprojlim_{n \geq 1} X_{n}(z) \subset \varprojlim_{n\geq 1} \Sigma_{M_{F_n}} = \Sigma, 
\] 
where the last equality follows 
from the closedness of $\Sigma$.  
By construction, we deduce immediately that $\Psi_{E}(x) = z$. 
In other words, $z \in \Lambda$ and thus $\Lambda$ is closed in $(\Sigma_E)^{\N^r}$ with respect to 
the prodiscrete topology. 
\par 
We conclude that $\Lambda$ is an admissible group  subshift of $(\Sigma_E)^{\N^r}$ and 
thus is a subshift of finite type by Theorem \ref{t:main-N-r-sft}. 
The proof is completed.  
\end{proof}

 \section{Main result}
\label{s:main-result}

We establish the following main result  whose proof is a natural generalization of the proof of \cite[Proposition~2]{kurka-97}:

\begin{theorem} 
\label{t:main-shadow}
Let $G$ be a countable monoid and let $A$ be an admissible Artinian group structure. 
Let  $\Sigma \subset A^G$ be an admissible Artinian group subshift. 
Suppose that $\Gamma \subset \End_{G\text{-grp}} (\Sigma)$  is a finitely generated abelian submonoid. 
Then the valuation action of $\Gamma$ on $\Sigma$ has the shadowing property.  
\end{theorem}

\begin{proof} 
Suppose that $S= \{ \tau_1, \cdots, \tau_r \}$ is a finitely generating set 
of the monoid $\Gamma$, i.e., $\Gamma = \langle \tau_1, \cdots, \tau_r \rangle$. 
For every $n \geq 1$, let $F_n \coloneqq  \{(a_1, \dots, a_r) \in \N^r \colon a_1, \cdots,a_r \leq n \}$. 
Then $(F_n)_{n \geq 1}$ is an increasing sequence of finite subsets of $\N^r$  
such that $\N^r  = \bigcup_{n \geq 1} F_n$. 
\par 
Since $G$ is a countable monoid, it admits an admissible exhaustion $(E_n)_{n \geq 0}$ of finite subsets of $G$ 
(cf.~Definition~\ref{d:admissible-sequence}). 
\par 
We then have a standard Hamming metric $d$ on $A^G$ associated with $(E_n)_{n \geq 0}$ 
by setting for every $x, y \in A^G$ (cf.~\eqref{d:hamming}): 
\begin{align}
\label{e:metric-main}
d(x,y) \coloneqq 2^{-n}, \mbox{ where }  n \coloneqq \sup \{ k \in \N \colon x\vert_{E_k} =y\vert_{E_k} \}. 
\end{align}

\par 
Let us fix $\varepsilon >0$ and choose an integer $n_0 >0$ such that 
$2^{-n_0} < \varepsilon$. 
Since  $\Sigma \subset A^G$ be an admissible Artinian group subshift and $E_{n_0}$ 
is finite, the restriction $\Sigma_{E_{n_0}}$ is a  admissible subgroup of $A^{E_{n_0}}$ 
and thus inherits a compatible admissible Artinian group structure.  
\par 
Therefore, with the notations as in Section~\ref{s:kurka-construction},  
 the column factorization 
$$
\Lambda \coloneqq \Lambda(\Sigma, E ;  \tau_1, \cdots, \tau_r) \subset (\Sigma_{E_{n_0}})^{\N^r}
$$ 
is an admissible Artinian group subshift of finite type 
by Theorem~\ref{l:closed-kurka-sft}. 
Hence, there exists an integer $N \geq 1$ such that 
\[
\Lambda = \Sigma( (\Sigma_{E_{n_0}})^{\N^r}; F_N, \Lambda_{F_N}).
\]  
\par 
Since $F_N$ is finite, Lemma~\ref{l:lipschitz-tau}  implies that there exists a finite constant $C \geq 1$ such that every cellular automaton 
$\tau_\alpha $ 
 where $\alpha  \in F_N$ 
 is $C$-Lipschitz, i.e., for every $x, y \in \Sigma$, we have 
\begin{equation} 
\label{e:c-lipschitz-commun}
d(\tau_\alpha(x), \tau_\alpha(y)) \leq C d(x,y).  
\end{equation}
\par 
Denote $\delta  \coloneqq  \frac{1}{2^{n_0}CNr}$.  
Let $(x_\tau)_{\tau \in \Gamma}$ be an $(S, d\vert_\Sigma,\delta)$-pseudo-orbit of the valuation action $T$ of $\Gamma$ on $\Sigma$. 
 Remark that $T$ is given by the evaluation map $T(\tau, x) \coloneqq \tau(x)$ for every $\tau \in \Gamma$ and $x \in \Sigma$. 
  \par    
Then by definition, we have  for all $\sigma \in S$ and $\tau \in \Gamma$ that 
\begin{equation} 
\label{e:delta-1}
d(T(\sigma ,x_\tau), x_{\sigma \tau}) < \delta.  
\end{equation}
\par 
Now let $\alpha= (a_1, \cdots, a_r)  \in F_N$ and $\tau \in \Gamma$. 
We have by the triangle inequality that 
\begin{align*} 
& d ( T ( \tau_{\alpha} ,x_\tau), x_{\tau_\alpha \tau} )  
 =  d ( \tau_1^{a_1}   \cdots    \tau_r^{a_r} (x_\tau), x_{\tau_\alpha \tau} )  \\
& \leq \sum_{k=0}^{a_1-1}  d (  \tau_1^k \tau_1(x_{\tau_1^{a_1-k-1}   \cdots    \tau_r^{a_r} \tau }), \tau_1^k  (x_{\tau_1^{a_1- k}   \cdots    \tau_r^{a_r} \tau }) ) \,\,  + 
\\
&  + \sum_{k=0}^{a_2-1}  d (  \tau_1^{a_1} \tau_2^k \tau_2(x_{\tau_1^{a_2-k-1}   \cdots    \tau_r^{a_r} \tau }), \tau_1^{a_1} \tau_2^k (x_{\tau_2^{a_2- k}   \cdots    \tau_r^{a_r} \tau }) )  \,\, + \\
&   \cdots   \\
&   + \sum_{k=0}^{a_r-1}  d (  \tau_1^{a_1} \cdots \tau_{r-1}^{a^{r-1}} \tau_r^k \tau_r(x_{\tau_r^{a_r-k-1} \tau }), \tau_1^{a_1} \cdots \tau_{r-1}^{a_{r-1}}\tau_r^k (x_{\tau_r^{a_r-k} \tau }) ). 
\end{align*}  
 \par
Therefore, it follows from the $C$-Lipschitz continuity of $\tau_{\beta}$ for every $\beta \in F_N$ (cf.~\eqref{e:c-lipschitz-commun})   
and from the choice of $\delta$ that:
\par 
\begin{align} 
\label{e:delta-epsilon-shadowing} 
d ( T ( \tau_{\alpha} ,x_\tau), x_{\tau_\alpha \tau} )  
& \leq C \sum_{k=0}^{a_1-1}  d (   \tau_1(x_{\tau_1^{a_1-k-1}   \cdots    \tau_r^{a_r} \tau }), x_{\tau_1^{a_1- k}   \cdots    \tau_r^{a_r} \tau })  \,\,  + 
  \\
&  + C \sum_{k=0}^{a_2-1}  d ( \tau_2(x_{\tau_1^{a_2-k-1}   \cdots    \tau_r^{a_r} \tau }), x_{\tau_2^{a_2- k}   \cdots    \tau_r^{a_r} \tau })   \,\, +   \nonumber \\
&   \cdots  && \nonumber \\
&   + C \sum_{k=0}^{a_r-1}  d (\tau_r(x_{\tau_r^{a_r-k-1} \tau }),  x_{\tau_r^{a_r-k} \tau }) \nonumber  \\ 
& \leq C (a_1 + \cdots + a_r) \delta  \nonumber\\
& \leq CNr \frac{1}{2^{n_0}CNr} = 2^{-n_0}.  \nonumber
\end{align}  
\par 
Consider  $z \in \Sigma_{E_{n_0}}^{\N^r}$ defined by 
$z(\alpha) \coloneqq x_{\tau_\alpha}\vert_{E_{n_0}}$ 
for every $\alpha \in \N^r$. 
We claim that $z \in \Lambda$. For this, let $\beta \in \N^r$, we must show that 
$z\vert_{\beta + F_N} \in \Lambda_{\beta + F_N}$. 
\par 
Indeed, by \eqref{e:delta-epsilon-shadowing} and by the choice of $n_0$, 
we have for every $\alpha \in F_N$ that 
\begin{align} 
  (\tau_{\alpha} (x_{\tau_\beta}))\vert_{E_{n_0}} =  x_{\tau_\alpha \tau_\beta} \vert_{E_{n_0}}=x_{\tau_{\alpha + \beta}}\vert_{E_{n_0}}. 
\end{align} 
\par 
Therefore, we infer from the definition of the subshift $\Lambda$ and from \eqref{e:psi-e} that 
$z\vert_{\beta + F_N} \in \Lambda_{\beta + F_N}$. 
This shows that $z \in \Lambda$. 
\par 
Hence, there exists a configuration $x \in \Sigma$ 
such that $ \Psi_{E_0} (x) = z$. 
We deduce from the definitions of $\Psi_{E_{n_0}}$ and $z$ 
that for all $\alpha \in \N^r$, we have 
\[
(\tau_\alpha(x))\vert_{E_{n_0}}= z(\alpha) = x_{\tau_\alpha}\vert_{E_{n_0}}. 
\]  
\par 
Let $\tau\in \Gamma$, then $\tau = \tau_\alpha$ 
for some $\alpha \in \N^r$.  
Then, we find that  
\[
d(T(\tau,x), x_\tau) = d(\tau_\alpha(x), x_{\tau_\alpha}) \leq 2^{-n_0} < \varepsilon, 
\]  
which implies that   
$x$  $\varepsilon$-shadows the $(S, d,\delta)$-pseudo-orbit $(x_\tau)_{\tau \in \Gamma}$ 
with respect to the standard Hamming metric $d$ (cf.~\eqref{e:metric-main}).  
Therefore, Lemma~\ref{l:intrinsic-def-shadow} 
and Definition~\ref{d:shadow-general} imply that 
the action of $\Gamma$ on $\Sigma$ has the shadowing property and the proof is completed.  
\end{proof}

We thus obtain the following direct consequence. 
 
\begin{corollary}
Let $G$ be a finitely generated abelian monoid   and let $A$ be an admissible Artinian group structure. 
Let $\Sigma$ be an admissible group subshift of $A^G$. 
Then the shift action $T$ of $G$ on $\Sigma$ has the shadowing property. 
\end{corollary}

\begin{proof}
The shift action of each $g \in G$ induces a map $\Sigma \to \Sigma$ defined by $x \mapsto g \star x$
which is clearly an admissible group cellular automaton. 
Since $G$ is a finitely generated abelian monoid, we can   conclude by Theorem~\ref{t:main-shadow} 
that the shift action of $G$ on $\Sigma$ has the shadowing property.  
\end{proof}

We now describe the following result which is sightly more general than Theorem~\ref{t:main-shadow} in  many cases of interest in practice. 
\par 
Let $G$ be a countable monoid and let $A$ be an Artinian group, resp. an Artinian module, resp. a compact Lie group, resp. an algebraic group over an alegebraically closed field.  
Let $\Sigma \subset A^G$ be a closed subshift such that $\Sigma_E$ is a subgroup, resp. a submodule, resp. a closed Lie subgroup, resp. an algebraic subgroup 
of $A^E$ 
for every finite $E \subset G$. \par 
Let $n \geq 1$ be an integer and 
consider a finite number of pairwise commuting $G$-equivariant maps $\tau_1, \cdots, \tau_n \colon \Sigma \to \Sigma$. 
Assume that for every $i =1, \cdots, n $, there exist a finite subset $M_i \subset G$ and a homomorphism $\mu_i \colon \Sigma_{M_i} \to A$ of groups, resp. of modules, resp. of Lie groups, resp. of algebraic groups 
such that 
\begin{equation} 
\label{e:main-shadow-subshift} 
\tau_i(x)(g) = \mu_i( (g\star x)\vert_{M_i})  \quad  \text{for all } x \in \Sigma \text{ and } g \in G. 
\end{equation}
Such maps $\mu_i$ might fail to extend to homomorphisms $A^{M_i} \to A$. 

\begin{theorem}
\label{t:main-shadow-subshif} 
With the above hypotheses and notations, let $\Gamma$ be the monoid generated by 
$\tau_1,\cdots, \tau_n$ with the binary operation given by composition of maps. Then  
the valuation action of $\Gamma$ on $\Sigma$ has the shadowing property. 
\end{theorem} 

\begin{proof} 
The proofs of Theorem~\ref{l:closed-kurka-sft} and 
Lemma~\ref{l:lipschitz-tau} 
can be modified in a  straightforward manner so that similar conclusions of 
Theorem~\ref{l:closed-kurka-sft} and 
Lemma~\ref{l:lipschitz-tau} are still valid for $\Gamma$ and for all $\tau \in \Gamma$ respectively.  
Therefore, the same proof of Theorem~\ref{t:main-shadow} can be applied to prove Theorem~\ref{t:main-shadow-subshif}. The easy verification is omitted. 
\end{proof}

\begin{proof}[Proof of Corollary~\ref{c:intro}] 
As in the proof of  \cite[Theorem~4.6]{csc-monoid-surj}, for every $\tau \in \Gamma$, there exists by the uniform continuity of $\tau$ a finite subset $M \subset G$ and a group (resp. module) homomorphism $\mu \colon \Sigma_M \to A$ such that 
\begin{equation} 
 \tau(x)(g) = \mu( (g\star x)\vert_M)  \quad  \text{for all } x \in \Sigma \text{ and } g \in G. 
\end{equation} 
Thus, Corollary \ref{c:intro} is an immediate consequence of Theorem~\ref{t:main-shadow-subshif}. 
\end{proof} 

\section{A counter-example}
The following example, inspired from \cite[Lemma~1]{salo-2018-polycyclic}, shows that our main results are optimal in the sense that we cannot remove the  hypothesis that $\Gamma$ is finitely generated in Corollary~\ref{c:intro} or in 
Theorem~\ref{t:main-shadow}. 

\begin{example}
\label{ex:counter} 
Let $A$ be an arbitrary group consisting of at least two elements. Let $G= \bigoplus_{\N} \Z / 2\Z$ be the direct sum indexed by $\N$ of copies of the group $ \Z / 2\Z$. 
Consider the subshift of constant configurations 
\[ 
\Sigma \coloneqq \{ a^G \colon a \in A\} \subset A^G. 
\] 
\par 
Then $\Sigma$ is clearly a subgroup of $A^G$ which is closed with respect to the prodiscrete topology. 
\par 
For every integer $n \geq 1$, 
let $E_n \coloneqq \{g = (g_i)_{i \in \N}  \in G \colon g_i = 0 \mbox{ for all } i \geq n  \}$ 
and let $E_0 \coloneqq \varnothing$. Then $(E_n)_{n \in \N}$ is an exhaustion of the group $G$. 
\par 
Denote by $d$ the induced Hamming metric on $A^G$ and fix $\varepsilon = \frac{1}{2}$. 
Remark that if $x, y \in \Sigma$ are such that $d(x, y) \leq \varepsilon$, then $x =y$. 
\par
We claim that for every integer $m \geq 1$, the right shift action of $G$ on $\Sigma$ does not satisfy 
the $(E_m, d\vert_\Sigma)$-shadowing property. 
Notably, it will follow that the above shift action does not have the shadowing property as defined in  \cite[Definition~2.2]{meyerovitch-pseudo-orbit}). 
\par 
To prove the claim, let $b \in A \setminus \{ 0\}$ and let $n \geq m \geq 1$ be integers. Consider the  $(E_m, d\vert_\Sigma, 2^{-n})$-pseudo-orbit $(x_g)_{g \in G}$ of $G$ in $\Sigma$ given by 
\begin{equation}
\label{eq:counter}
x_g = 0^G \mbox{ if } g \in E_n \quad \mbox{ and } \quad  x_g = b^G \mbox{ if  } g \in G\setminus E_n.  
\end{equation}
\par
Suppose that there exists $x \in \Sigma$ which $\varepsilon$-shadows $(x_g)_{g \in G}$. Then for every $g \in G$, 
we have 
$\varepsilon \geq d(g\star x, x_g)$. Hence, $g \star x = x_g$ and as $x$ is constant, it follows that $x=g\star x = x_g$ for every $g \in G$. 
Since $E_n$ and $G \setminus E_n$  are nonempty, \eqref{eq:counter} then implies that $0^G=b^G$, which is a contradiction since $b \neq 0$. Hence, no point $x \in \Sigma$ can $\varepsilon$-shadow $(x_g)_{g \in G}$ and the claim is proved.

\end{example}

\bibliographystyle{siam}
\bibliography{lsalgsft}

\end{document}